\documentclass[11pt]{article}
\usepackage{IEEEtrantools}
\usepackage{mathtools, nccmath}
\usepackage{latexsym}
\usepackage{amsfonts}
\usepackage{amssymb}
\usepackage{psfrag}
\usepackage{graphicx}
\usepackage{epsfig}
\usepackage{amsmath,amsfonts,amsthm}
\usepackage{mathrsfs}
\usepackage{enumerate}
\usepackage{dsfont}
\usepackage{hyperref}
\usepackage{pst-node}
\usepackage{cleveref}
\usepackage{enumitem}
\usepackage{tikz-cd}
\usepackage{multirow}
\usepackage{tikz}
\usetikzlibrary{arrows,decorations.markings}
\usepackage{bookmark}
\usepackage{hyperref}
\hypersetup{
     colorlinks   = true,
     citecolor    = blue
}

\newcommand\restr[2]{{
  \left.\kern-\nulldelimiterspace 
  #1 
  \littletaller 
  \right|_{#2} 
  }}

\newcommand{\littletaller}{\mathchoice{\vphantom{\big|}}{}{}{}}

\newcommand{\be}{\begin{equation}}
\newcommand{\ee}{\end{equation}}
\newcommand{\bea}{\begin{eqnarray}}
\newcommand{\eea}{\end{eqnarray}}
\newcommand{\bean}{\begin{eqnarray*}}
\newcommand{\eean}{\end{eqnarray*}}
\newcommand{\brray}{\begin{array}}
\newcommand{\erray}{\end{array}}
\newcommand{\biearray}{\begin{IEEEarray}{rCl}}
\newcommand{\eiearray}{\end{IEEEarray}}

\renewcommand{\theequation}{\thesection.\arabic{equation}}

\newtheorem{dfn}{Definition}[section]
\newtheorem{thm}[dfn]{Theorem}
\newtheorem{lmma}[dfn]{Lemma}
\newtheorem{ppsn}[dfn]{Proposition}
\newtheorem{crlre}[dfn]{Corollary}
\newtheorem{xmpl}[dfn]{Example}
\newtheorem{rmrk}[dfn]{Remark}

\newcommand{\bdfn}{\begin{dfn}\rm}
\newcommand{\bthm}{\begin{thm}}
\newcommand{\blmma}{\begin{lmma}}
\newcommand{\bppsn}{\begin{ppsn}}
\newcommand{\bcrlre}{\begin{crlre}}
\newcommand{\bxmpl}{\begin{xmpl}}
\newcommand{\brmrk}{\begin{rmrk}\rm}

\newcommand{\edfn}{\end{dfn}}
\newcommand{\ethm}{\end{thm}}
\newcommand{\elmma}{\end{lmma}}
\newcommand{\eppsn}{\end{ppsn}}
\newcommand{\ecrlre}{\end{crlre}}
\newcommand{\exmpl}{\end{xmpl}}
\newcommand{\ermrk}{\end{rmrk}}

\def \qed { \mbox{}\hfill
$\Box$\vspace{1ex}}

\title{Sections and Chapters}

\begin{document}
	
	
	\author{\sc{Keshab Chandra Bakshi, Silambarasan C, Biplab Pal}}

	\title{Regular irreducible inclusions of simple $C^*$-algebras and crossed product structure}

	\maketitle
	
	
\begin{abstract}
We study regular irreducible inclusions $B\subset A$ of simple unital $C^*$-algebras admitting a conditional expectation. We introduce a generalized notion of quasi-basis extending Watatani’s framework and show that such inclusions admit a unitary orthonormal generalized quasi-basis. As a consequence, we prove that every regular irreducible inclusion in this setting is canonically isomorphic to a reduced twisted crossed product of $B$ by its Weyl group. This extends earlier crossed product characterizations beyond the finite-index setting.
\end{abstract}

	\bigskip
	
	{\bf AMS Subject Classification No.:} {\large 46}L{\large 37}\,, {\large 47}L{\large 40}\,, {\large 46}L{\large 05}\,, {\large 43}A{\large 30}\,.
	
	\bigskip
	\hypersetup{linkcolor=blue}

\section{Introduction}

Inclusions of C$^*$-algebras provide a natural framework for studying symmetry via normalizers and crossed product constructions. In the finite-index setting, regular irreducible inclusions of simple unital C$^*$-algebras admitting a conditional expectation are known to admit crossed product descriptions by discrete groups (see \cite{BG2}). The corresponding result is well-known for factors: irreducible regular subfactors with a conditional expectation admit crossed product decompositions by discrete groups without any finite-index assumption.

In this paper, we remove the finite-index assumption in the C$^*$-algebraic setting. Let $B \subset A$ be an inclusion of simple unital C$^*$-algebras with a conditional expectation. Recall, the inclusion is called \emph{regular} if the unitary normalizers
\[
{\mathcal{N}}_A(B)=\{u \in \mathcal{U}(A):uBu^*=B\}
\]
generate $A$ as a $C^*$-algebra.

We introduce a generalized quasi-basis extending Watatani's notion via the Bures topology, and prove that every regular inclusion admits a unitary orthonormal generalized quasi-basis.

As a consequence, we obtain a structural characterization: if $B \subset A$ is regular and irreducible, then its Weyl group $W={\mathcal{N}}_A(B)/\mathcal{U}(B)$ admits an outer cocycle action on $B$, and
\[
A \cong B \rtimes_r^{(\alpha,\sigma)} W.
\]

Thus every regular irreducible inclusion arises as a reduced twisted crossed product, extending the finite-index theory. In the context of subfactors, a similar result was obtained by Choda in \cite{Ch2}, building on earlier ideas developed in \cite{Ch}.

We begin with a review  on the Bures topology and reduced twisted crossed products in Section 2. In Section 3, we introduce generalized quasi-bases and establish their existence for regular inclusions. Finally, we prove the crossed product characterization for irreducible inclusions.
\color{black}
\section{Preliminaries}\label{preliminaries}
This section summarizes some basic concepts that will be needed in the sequel. Our treatment is concise, and we refer the reader to the literature for more comprehensive accounts.
\subsection{Bures topology}
We recall the topology introduced by Bures \cite{Bu}, which was subsequently identified by Mercer \cite{Me} as the appropriate framework for studying convergence of Fourier series in crossed products. In this setting, we consider an inclusion $N \subseteq M$ of von Neumann algebras. We begin with the following definition.
\begin{dfn}[\cite{Bu}]
Let $N \subset M$ be an inclusion of von Neumann algebras, an $N$-inner-product on $M$ is a mapping $(x,y) \mapsto [x \mid y]$ from $M \times M$ to $N$ which satisfies the followings, for all $x,y,z\in M$ and all $n,n' \in N$:
\begin{enumerate}
\item[(i)] $[x \mid yn] = [ x \mid y]n$, and 
\quad 
$[xn \mid y] = n^*[x \mid y].
$

\item[(ii)] $[x + y \mid z] = [x \mid z] + [y \mid z].$

\item[(iii)] $[y \mid x] = [x \mid y]^*.$

\item[(iv)] $[x \mid x] > 0 
\quad \text{if } x \neq 0.$

\item[(v)] $[n \mid n'] = n^*n'.$

\item[(vi)] $[xy \mid z] = [y \mid x^*z].$

\item[(vii)] 
For fixed $y \in M$, the mapping $x \mapsto [x \mid y]$
from $M$ to $N$ is ultraweakly continuous.
\end{enumerate}

If condition (iv) is weakened to
$\text{(iv)'} \quad [x \mid x] \ge 0 
\quad \text{for all } x \in M,$
we call $[\cdot \mid \cdot]$ an $N$-semi-inner-product.
\end{dfn}
For each normal state $\omega$ on $N$, define a seminorm $\|\cdot\|_\omega$ on $M$ by 
$\|x\|_\omega^2 = \omega([x \mid x])$ for all $x \in M$. The topology on $M$ induced by the family of seminorms $\{\|\cdot\|_\omega : \omega \in N_{*}\}$ is called the $N$-topology on $M$. It is clear that the $N$-topology is Hausdorff if and only if $[\cdot \mid \cdot]$ defines an $N$-valued inner product.

Let $E \colon M \to N$ be a conditional expectation. Define a map $[\cdot \mid \cdot] \colon M \times M \to N$ by $[x \mid y] = E(x^*y)$. Then $[\cdot \mid \cdot]$ defines an $N$-valued inner product on $M$. The topology on $M$ induced by this inner product will be referred to as the \emph{Bures topology}. For further details on the Bures topology, see \cite{CaSm}.

\begin{dfn}[\cite{Bu}]\label{impformainthm2}
A family $(v_i)$ of partial isometries of $M$ will be called 
orthonormal if\[
[v_i \mid v_j] = 0 
\qquad \text{if } i \ne j,
\]

\[
[v_i \mid v_i] = v_i^* v_i
\qquad \text{if } i = j.
\]
We call such a family an orthonormal basis for $M$ 
if the finite linear combinations $\sum_i v_i n_i,$ with each $n_i \in N, $ are dense in $M$ in the $N$-topology.
\end{dfn}
\begin{ppsn}[\cite{Bu}]\label{impformainthm3}
Let $N \subset M$ be an inclusion of von Neumann algebras, and let 
$[\cdot \mid \cdot]$ be an $N$-inner-product on $M$. 
Suppose that $(v_i)$ is an orthonormal basis for $M$. 
Then every element $x \in M$ admits a decomposition of the form
\[
x = \sum_i v_i [v_i \mid x]\,,
\]
where the sum converges in the $N$-topology.
\end{ppsn}
\subsection{Reduced twisted crossed product}\label{Reduced twisted crossed product}
Recall that a discrete twisted $C^*$-dynamical system is a quadruple $(A, G,
\alpha, \sigma)$ consisting of a unital $C^*$-algebra $A$, a 
discrete group $G$, a map $\alpha: G \to \mathrm{Aut}(A)$ and a map
$\sigma: G\times G \to \mathcal{U}(A)$ satisfying the following identities:
\begin{eqnarray}
    &&\alpha_g \circ \alpha_h = \sigma(g, h ) \alpha_{gh} \sigma(g,h)^*;\\\label{impincocycledefn}
    &&\sigma (g, h) \sigma (gh, k) = \alpha_g(\sigma(h,k)) \sigma(g, hk);
\text{  and,}\\
&&\sigma (g, e) = \sigma( e, g) = 1
\end{eqnarray}
for all $g, h, k \in G$. Such a $\sigma$ is called a normalized
$\mathcal{U}(A)$-valued $2$-cocycle on $G$, and an $\alpha$ as above is called a
twisted action of $G$ on $A$ with respect to the cocycle $\sigma$. 
Assume that $A$ is represented faithfully and non-degenerately as a $C^*$-algebra
of bounded operators on a Hilbert space $H$. The reduced twisted
crossed product $A \rtimes^r_{(\alpha, \sigma)} G$ (also denoted by $C^*_r(A, G, \alpha, \sigma)$) is then realized as the $C^*$-subalgebra
of $\mathcal{B}(\ell^2(G,H))$ generated by
$\pi_\alpha(A)$ and $\lambda_\sigma(G)$. Here, $\pi_\alpha$ denotes the representation of $A$ on $\ell^2(G,H)$ given by
\[
(\pi_\alpha(a)\xi)(g) = \alpha_{g^{-1}}(a)\,\xi(g),
\quad a \in A,\ \xi \in \ell^2(G,H),\ g \in G,
\]
for each $g \in G$, the unitary $\lambda_g$ on
$\ell^2(G,H)$ is defined by
\[
(\lambda_g\xi)(h) = \sigma(h^{-1}, g)\,\xi(g^{-1}h),
\quad \xi \in \ell^2(G,H),\ h \in G.
\]
We have the following covariant representation of
the twisted dynamical system $(A, G, \alpha, \sigma)$, as follows:
\[
\pi_\alpha(\alpha_g(a)) = \operatorname{Ad}(\lambda_g)(\pi_\alpha(a)),
\]
and
\[
\lambda_g\lambda_h
= \pi_\alpha(\sigma(g,h))\,\lambda_{gh},
\quad a \in A,\ g,h \in G.\]
Moreover, as observed in \cite[p. 552]{Qu} (see also
\cite[Remark 3.12]{PR}), the resulting $C^*$-algebra
$A \rtimes^r_{(\alpha, \sigma)} G$ does not depend, up to isomorphism,
on the particular choice of the Hilbert space $H$. For more on reduced twisted crossed product, we refer the reader to \cite{Be, BeO}.
\begin{rmrk}\label{impinmainthm6}
We record here some properties of the twisted dynamical system $(A, G, \alpha, \sigma)$ that will be used subsequently:
\begin{enumerate}
\item[$(i)$] The condition in \Cref{impincocycledefn} can be expressed equivalently as (see \cite{Be}[eqs.\ (4) and (5)]): for all $g,h,k \in G$,
\[
\alpha_g\big(\sigma(h,k)\big)^*\,\sigma(g,h)
= \sigma(g,hk)\,\sigma(gh,k)^*.
\]

\item[$(ii)$] For each $g \in G$, one has (see \cite{Be}[eq.\ (8)]):
\[
\lambda_g^* = \pi_\alpha\big(\sigma(g^{-1},g)\big)^*\,\lambda_{g^{-1}}.
\]

\item[$(iii)$] For all $g,h \in G$, the following identity holds (see \cite{Be}[eq.\ (9)]):
\[
\lambda_g\,\lambda_h\,\lambda_g^*
= \pi_\alpha\big(\sigma(g,h)\,\sigma(ghg^{-1})^*\big)\,\lambda_{ghg^{-1}}.
\]

\end{enumerate}
\end{rmrk}


\section{Structure of regular inclusions of simple $C^*$-algebras beyond finite index}\label{Structure of regular inclusions}

In this section, we study regular irreducible inclusions $B \subset A$ of simple unital $C^*$-algebras equipped with a conditional expectation $E : A \to B$, without imposing any finiteness assumptions such as finite index. Throughout, we assume that \( E \) is faithful. Our first objective is to establish the existence of a unitary orthonormal generalized quasi-basis for $E$, in a sense to be specified below. We then undertake a structural analysis of such inclusions. In particular, we show that the associated Weyl group $W$ admits an outer cocycle action $(\alpha,\sigma)$ on $B$, and that $A$ is isomorphic to the reduced twisted crossed product $B \rtimes^r_{(\alpha,\sigma)} W$.

\medskip

Recall that, given an inclusion $B \subset A$ of unital $C^*$-algebras with a common identity, the unitary normalizer of $B$ in $A$ is defined by $\mathcal{N}_A(B) = \{u \in \mathcal{U}(A) : u B u^* = B\}.$ The inclusion $B \subset A$ is said to be \emph{regular} if $\mathcal{N}_A(B)$ generates $A$ as a $C^*$-algebra. A natural class of examples arises from reduced twisted crossed products. Indeed, if $A\rtimes^r_{(\alpha,\sigma)}G$ is a reduced twisted crossed product as described in \Cref{Reduced twisted crossed product}, then the inclusion $A\subset A\rtimes^r_{(\alpha,\sigma)}G$ is regular. Furthermore, if $A$ is simple and the action is outer, then \cite[Theorem 3.2]{Be} implies that $A\rtimes^r_{(\alpha,\sigma)}G$ is simple, while \cite[Theorem 5.1]{BeO} shows that the above inclusion is irreducible

Now let $B \subset A$ is an inclusion of simple unital $C^*$-algebras, the unitary group $\mathcal{U}(B)$ forms a normal subgroup of $\mathcal{N}_A(B)$. In analogy with the notion of the Weyl group for inclusions of von Neumann algebras (see \cite[Definition~2.11]{BG1}), we recall the following definition:

\begin{dfn}\label{weyl-defn}
Let $B \subset A$ be an inclusion of unital $C^*$-algebras with common unit. The \emph{Weyl group} associated with the inclusion $B \subset A$ is the quotient group \( \frac{\mathcal{N}_A(B)}{\mathcal{U}(B) } \) and will be denoted by $W(B\subset A)$.
\end{dfn}

The following lemma will play an important role.

\begin{lmma}[\cite{BG2}]\label{impformainthm}
Let $B \subset A$ be an irreducible unital inclusion of simple $C^*$-algebras, and let $w \in \mathcal{N}_A(B) \setminus \mathcal{U}(B) $ be a unitary, then $\mathrm{Ad}_w$ defines an outer automorphism of $B$. Moreover, for any $u, v \in \mathcal{N}_A(B)$ and for any conditional expectation $E:A \to B$, one has $E(vu^*) = 0 = E(v^*u)$ whenever $[u] \neq [v]$ in $W(B \subset A)$.
\end{lmma}
  \subsection{Unitary orthonormal generalized quasi-basis for regular inclusions}\label{Unitary orthonormal generalized quasi-basis}
  It was conjectured in \cite{BG3} and subsequently proved in \cite{CKP} that any finite-index regular inclusion of type $II_1$ factors admits a unitary orthonormal basis. In the $C^*$-algebraic setting, it was shown in \cite{BG2}[Corollary 3.10] that for any regular inclusion $B \subset A$ of simple unital $C^*$-algebras with a finite-index conditional expectation, the minimal conditional expectation $E_0$ admits a unitary orthonormal quasi-basis. Motivated by these results, we consider arbitrary regular irreducible inclusions of simple unital $C^*$-algebras equipped with a conditional expectation, not necessarily of finite index. Our goal is to obtain an analogous result in this setting. To this end, we introduce a generalized notion of quasi-basis, extending Watatani's definition of a quasi-basis (\cite{W}), and show that every such inclusion admits a generalized quasi-basis.
  
  Recall that for any $C^*$-algebra $A$, if $\pi_S$ is the universal representation of $A$, then the enveloping von Neumann algebra is the SOT-closure of $\pi_S(A)$, and is denoted by $A''$. Its double dual $A^{**}$ is isometrically isomorphic to the enveloping von Neumann algebra $A''$ (see \cite{Pe}[Proposition~3.7.8]). 
  
  Also recall that for an inclusion $B \subset A$ of $C^*$-algebras with a conditional expectation $E : A \to B$, $B''$ is naturally a von Neumann subalgebra of $A''$ (see \cite{Pe}[Corollary~3.7.9]), and $E$ extends to a normal conditional expectation $E'' : A'' \to B''$ (see \cite{Li}[Lemma~4.1.4]).

\begin{dfn}\label{Generalized}
Let $B \subset A$ be an inclusion of $C^*$-algebras and 
let $E : A \to B$ be a conditional expectation. 
Denote by $E^{''} : A'' \to B''$ the normal extension 
of $E$ to the enveloping von Neumann algebra. A set $\{\lambda_i\}_{i \in I} \subset A$ is called a 
\emph{generalized (right) quasi-basis} for $E$ if
\[
x = \sum_{i \in I}\lambda_i\, E(\lambda_i^*\,x)
\quad \text{for all } x \in A,
\]
where the sum converges in the Bures topology on $A''$ with respect to the normal conditional expectation $E''$. 
\end{dfn} 
A generalized (right) quasi-basis is said to be \emph{orthonormal} if $E(\lambda_i^*\lambda_j)=\delta_{i,j}$ and \emph{unitary} if all its elements are unitary. Related constructions in the setting of von Neumann algebras were studied in \cite{FI,Po}. Now let $W$ denote the Weyl group of the
inclusion $B \subset A$ with a fixed set of left coset
representatives $\{u_g : g \in W\}$ in $\mathcal{N}_A(B)$, with $u_e =
1$. Then we have the following:
\begin{thm}\label{mainthm}
Let $B \subset A$ be a regular irreducible unital inclusion of simple $C^*$-algebras, equipped with a conditional expectation $E : A \to B$, then $\{u_g: g \in W\}$ is a unitary orthonormal generalized (right) quasi-basis for the conditional expectation $E$.
 \end{thm}
\begin{prf}
Let $C_0 := \bigoplus_{g \in W} u_g B$ denote the algebraic direct sum. Then every element $x \in C_0$ can be written in the form $x = \sum_{g \in F} u_g b_g,$ where $F \subset W$ is a finite subset and $b_g \in B$ for each $g \in F$. Set $C := \overline{C_0}^{\,B}$ is the closure of $C_0$ in the Bures topology on $A''$ with respect to the normal conditional expectation $E''$. Clearly, $\mathrm{span}\{\mathcal{N}_A(B) \}\subset C \subset A''.$ We claim that $\overline{\mathrm{span}\{\mathcal{N}_A(B)\}}^{\|\cdot\|} \subset C.$ Let $x \in \overline{\mathrm{span}\{\mathcal{N}_A(B)\}}^{\|\cdot\|}$. Then there exists a sequence $(x_n) \subset \mathrm{span}\{\mathcal{N}_A(B)\}$ such that $x_n \xrightarrow{\|\cdot\|} x.$ Since the $\sigma$-strong topology is weaker than the norm topology, it follows that $(x_n - x) \to 0$ in the $\sigma$-strong topology. Consequently,
\[
(x_n - x)^*(x_n - x) \to 0
\]
in the $\sigma$-weak topology. Applying the normal conditional expectation $E$, we obtain
\[
E\big((x_n - x)^*(x_n - x)\big) \to 0
\]
in the $\sigma$-weak topology. Hence, for every normal linear functional $\omega$ on $B''$,
\[
\omega\!\left(E\big((x_n - x)^*(x_n - x)\big)\right) \to 0.
\]
This shows that $x_n \to x$ in the Bures topology on $A''$ with respect to the normal conditional expectation $E''$. Since $C$ is Bures closed, we conclude that $x \in C$. Hence $\overline{\mathrm{span}\{\mathcal{N}_A(B)\}}^{\|\cdot\|} \subset C$. Since $B \subset A$ is regular, we have $A \subset C \subset A''.$ We now claim that $A'' \subset C$. Let $x \in A''$. By Kaplansky's density theorem, there exists a net 
$\{x_i\} \subset A$ with $\|x_i\| \le \|x\|$ for all $i$ such that 
$x_i \to x$ in the strong operator topology. Consequently,
\[
(x_i - x)^*(x_i - x) \to 0
\]
in the weak operator topology. Since on norm-bounded sets the weak operator topology coincides with the $\sigma$-weak topology, it follows that
\[
(x_i - x)^*(x_i - x) \to 0
\]
in the $\sigma$-weak topology. Applying the normal conditional expectation $E$, we obtain
\[
E\big((x_i - x)^*(x_i - x)\big) \to 0
\]
in the $\sigma$-weak topology. Therefore, $x_i \to x$ in the Bures topology on $A''$ with respect to the normal conditional expectation $E''$. Since $C$ is Bures closed and $x_i \in A \subset C$, we conclude that $x \in C$. Hence, $A'' \subset C$, and therefore $C = A''$. 
Invoking Lemma \ref{impformainthm}, we observe that the family 
$\{u_g : g \in W\}$ satisfies all the properties listed in \Cref{impformainthm2}. 
Consequently, it forms an orthonormal basis for $A''$. By Proposition \ref{impformainthm3}, every element $x \in A''$ (and in particular, every $x \in A$) admits a decomposition of the form
\[
x = \sum_{g \in W} u_g\, E(u_g^* x).
\]
It follows that $\{u_g:g\in W\}$ forms a generalized (right) quasi-basis for $E$. This completes the proof.
\qed
\end{prf}

Let $G$ be a discrete group, $H$ a subgroup of $G$, and $\sigma$ a $2$-cocycle on $G$ and $A$ is simple and suppose there is a twisted outer action of $G$ on $A$ with respect to the cocycle $\sigma$. We obtain the following result, motivated by \cite{Te}.

\begin{crlre}
Suppose that the inclusion $A \rtimes^r_{(\alpha,\sigma)} G \subset A \rtimes^r_{(\alpha,\sigma)} H$ is irreducible. Assume further that there exists a conditional expectation $E: A \rtimes^r_{(\alpha,\sigma)} G \to A \rtimes^r_{(\alpha,\sigma)} H.$ Then the inclusion is regular if and only if $H$ is a normal subgroup of $G$.
\end{crlre}
\begin{prf}
 First, assume that the inclusion 
$A \rtimes^r_{(\alpha,\sigma)} H \subset A \rtimes^r_{(\alpha,\sigma)} G$ 
is regular. Then, by Proposition \ref{mainthm}, there exists a generalized quasi-basis $\{u_i\}$ in $\mathcal{N}_{A \rtimes^r_{(\alpha,\sigma)} G}(A \rtimes^r_{(\alpha,\sigma)} H)$ for the conditional expectation $E$. In particular for any $g\in G$, we can write
$\lambda_g^* = \sum_{i \in W} u_i\, E(u_i^* \lambda_g^*)$,
where $W$ denotes the Weyl group associated with the inclusion, and the series converges in the Bures topology on $(A \rtimes^r_{(\alpha,\sigma)} G)^{''}$ with respect to the normal conditional expectation $E''$. Hence, there exists an index $i_0$ such that $E(u_{i_0}^* \lambda_g^*) \neq 0$, and consequently $E(\lambda_g u_{i_0}) \neq 0$. Proceeding as in \cite{Te}, it follows that $\lambda_g u_{i_0} \in A \rtimes^r_{(\alpha,\sigma)} H$. Therefore,
$\lambda_g (A \rtimes^r_{(\alpha,\sigma)} H) \lambda_g^* = A \rtimes^r_{(\alpha,\sigma)} H$
for all $g \in G$. In particular, for any $g \in G$ and $h \in H$, we have
$\lambda_g \lambda_h \lambda_g^* \in A \rtimes^r_{(\alpha,\sigma)} H$. By Remark \ref{impinmainthm6}$(iii)$, we know that
$\lambda_g \lambda_h \lambda_g^* = \sigma(g,h)\sigma(ghg^{-1}, g)^* \lambda_{ghg^{-1}}$. It follows that $ghg^{-1} \in H$ for all $g \in G$ and $h \in H$. Hence, $H$ is a normal subgroup of $G$.

Conversely, suppose that $H$ is a normal subgroup of $G$. Then by again using Remark \ref{impinmainthm6}$(iii)$, for any $g \in G$, it is straightforward to verify that 
$\lambda_g \in \mathcal{N}_{A \rtimes^r_{(\alpha,\sigma)} G}(A \rtimes^r_{(\alpha,\sigma)} H)$. 
Consequently, we have the following
 $\overline{\mathrm{span}\{\mathcal{N}_{A \rtimes^r_{(\alpha,\sigma)} G}(A \rtimes^r_{(\alpha,\sigma)} H)\}}^{\|\cdot\|}
= A \rtimes^r_{(\alpha,\sigma)} G$.
Hence, the inclusion 
$A \rtimes^r_{(\alpha,\sigma)} H \subset A \rtimes^r_{(\alpha,\sigma)} G$ 
is regular.
\qed
\end{prf}

\begin{rmrk}
    It is worth noting that, in \cite{BeO}, assuming that \(H\) is a normal subgroup of \(G\), the $C^*$-irreducibility, in the sense of \cite{Ro}, of the inclusion
$A \rtimes^r_{(\alpha,\sigma)} H \subset A \rtimes^r_{(\alpha,\sigma)} G$
was investigated.
\end{rmrk}
\bigskip


\subsection{Charecterization of regular irreducible inclusions of simple $C^*$-algebras}\label{Charecterization of regular irreducible inclusions}
In \cite[Theorem 4]{Ch2}, building on techniques developed earlier in \cite[Theorem 7]{Ch}, Choda proved that for any factor $M$ with separable predual, and for every irreducible regular subfactor $N \subset M$ admitting a faithful conditional expectation from $M$ onto $N$, there exists a countable discrete group $G$ together with an outer cocycle action $(\sigma,\omega)$ on $N$ such that 
$M \cong N \rtimes_{(\sigma,\omega)} G$. In the $C^*$-algebraic setting, it was later shown in \cite{BG2} that if $B \subset A$ is a regular 
irreducible  inclusion of simple unital $C^*$-algebras with a finite index conditional expectation from $A$
onto $B$, then its Weyl group $W$ admits an outer cocycle action $(\alpha, \sigma)$ on $B$ such that $(B
\subset A) \cong (B \subset B \rtimes^r_{(\alpha, \sigma)} W)$. In the present article, we aim to generalize this result to the case where the conditional expectation is not assumed to have finite index. We first record the following observation:

\begin{ppsn}
  Let $B \subset A$ be a regular 
irreducible  unital inclusion of simple $C^*$-algebras with a conditional expectation from $A$
onto $B$.  Then the associated Weyl group $W$ is discrete with respect to the quotient topology induced by the norm topology on $A$.  
\end{ppsn}
\begin{prf}
Endow $\mathcal{N}_A(B)$ with the subspace topology induced by the norm topology on $A$. Then it is straightforward to verify that $\mathcal{N}_A(B)$ is a topological group. Moreover, $\mathcal{U}(B)$ is a closed normal subgroup of $\mathcal{N}_A(B)$. Hence, the Weyl group $W = \frac{\mathcal{N}_A(B)}{\mathcal{U}(B)
   }$ is a topological group with respect to the quotient topology. Furthermore, this topology agrees with the metric topology induced by the metric 
$\tilde{d}([u],[v]) = \inf_{w \in \mathcal{U}(B)} \|u - vw\|$. Now let $[u] \neq [v]$ in $W$. Then, for any $w \in \mathcal{U}(B)$, we have
$\|u - vw\|^2 \geq \|E((u - vw)^*(u - vw))\|$.
Expanding the expression inside $E$, we obtain
$E((u - vw)^*(u - vw)) = E(1 - w^*v^*u - u^*vw + 1)$.
By Lemma \ref{impformainthm}, it follows that
$\|E((u - vw)^*(u - vw))\| = 2$.
Hence, $\|u - vw\| \geq \sqrt{2}$ for all $w \in \mathcal{U}(B)$. Therefore, $\tilde{d}([u],[v]) \geq \sqrt{2}$ whenever $[u] \neq [v]$, which shows that $W$ is discrete.
\qed
\end{prf}
\begin{thm}\label{characterization-irreducible-regular}
  Let $B \subset A$ be a regular 
irreducible  unital inclusion of simple $C^*$-algebras with a conditional expectation from $A$
onto $B$.  Then, its Weyl group $W$ admits an outer cocycle action $(\alpha, \sigma)$ on $B$ such that $(B
\subset A) \cong (B \subset B \rtimes^r_{(\alpha, \sigma)} W)$.
\end{thm}
\begin{prf}
Let $\{u_g : g = [u_g] \in W\}$ be a fixed choice of (left) coset representatives of $W$ in $\mathcal{N}_A(B)$. Define $\alpha_g : B \to B$ by $\alpha_g(b) = u^*_{g^{-1}} b u_{g^{-1}}$. Then each $\alpha_g$ is an automorphism of $B$. We claim that the map $\alpha : W \to \mathrm{Aut}(B)$, given by $g \mapsto \alpha_g$, defines a cocycle action with respect to a $\mathcal{U}(B)$-valued cocycle $\sigma$, which we now construct.

Observe that $[u_g u_h] = [u_{gh}]$ for all $g,h \in W$. Therefore, there exists a function 
\begin{equation} \label{impinmainthm3}
\sigma : W \times W \to \mathcal{U}(B) \quad \text{ such that} \quad
u_g u_h = u_{gh} \sigma(h^{-1}, g^{-1})^*.
\end{equation}
It is a straightforward verification that $(\alpha, \sigma)$ satisfies the following identities for all $g, h, k \in W$:
\[
\alpha_g \circ \alpha_h = \sigma(g, h ) \alpha_{gh} \sigma(g,h)^*;\]
\[
\sigma (g, h) \sigma (gh, k) = \alpha_g(\sigma(h,k)) \sigma(g, hk);
\text{  and,}\]
\[
\sigma (g, e) = \sigma( e, g) = 1.
\]

Let $S(B)$ denote the state space of $B$. For each $\phi \in S(B)$, define $\tilde{\phi}$ on $A$ by $\tilde{\phi}(x) = \phi \circ E(x)$. Let $(\pi_{\tilde{\phi}}, H_{\tilde{\phi}})$ be the GNS representation of $A$ associated with $\tilde{\phi}$, and let $\{\eta^{\tilde{\phi}}_a : a \in A\}$ denote the corresponding dense subset of $H_{\tilde{\phi}}$. Define $H_U = \bigoplus_{\phi \in S(B)} H_{\tilde{\phi}}$ and 
$\pi_U = \bigoplus_{\phi \in S(B)} \pi_{\tilde{\phi}} : A \to B(H_U)$. 
Then $\pi_U$ is a faithful representation of $A$ by \cite{Pe}[Proposition 3.7.4], since $E$ is faithful and the set $\{\phi \circ E : \phi \in S(B)\}$ is a separating family of states of the $C^*$-algebra $A$.

Consider the subspace $L \subset H_U$ given by
$$L = \overline{\operatorname{span} \left\{ (\eta^{\tilde{\phi}}_{b})_{\phi \in S(B)} : b \in B,\ (\eta^{\tilde{\phi}}_{b}) \text{ is supported only at the } \phi\text{-coordinate} \right\}}.$$
Then for each $b \in B$, we have $\pi_U(b)L \subset L$, and $(\pi_U|_B, L)$ is a faithful non-degenerate representation of $B$. We now claim that 
$H_U = \bigoplus_{g \in W} \pi_U(u_g)L$. 

To prove this, we first show that if $g \neq h$, then $\pi_U(u_g)L \perp \pi_U(u_h)L$. It suffices to verify this for vectors in $L$ of the form $(\eta^{\tilde{\phi}}_{b})$, supported only at the $\phi$-coordinate, and $(\eta^{\tilde{\psi}}_{c})$, supported only at the $\psi$-coordinate. We compute:
\begin{eqnarray}\nonumber
\big\langle \pi_U(u_g)(\eta^{\tilde{\phi}}_{b}), \, \pi_U(u_h)(\eta^{\tilde{\psi}}_{c}) \big\rangle
&=& \delta_{\phi,\psi} \, \big\langle \pi_{\tilde{\phi}}(u_g)\eta^{\tilde{\phi}}_{b}, \, \pi_{\tilde{\phi}}(u_h)\eta^{\tilde{\phi}}_{c} \big\rangle.
\end{eqnarray}
If $\phi \neq \psi$, the inner product is zero. If $\phi = \psi$, then
\begin{eqnarray}\nonumber
\big\langle \pi_U(u_g)(\eta^{\tilde{\phi}}_{b}), \, \pi_U(u_h)(\eta^{\tilde{\phi}}_{c}) \big\rangle
&=& \tilde{\phi}(c^* u_h^* u_g b) \\
\nonumber
&=& \phi \circ E(c^* u_h^* u_g b) \\
\nonumber
&=& \phi\big(c^* E(u_h^* u_g) b\big) \\
\nonumber
&=& 0 \quad \text{(by Lemma \ref{impformainthm})}.
\end{eqnarray}
Thus, $\pi_U(u_g)L \perp \pi_U(u_h)L$ whenever $g \neq h$.

Now we claim that  $H_U \subset \bigoplus_{g \in W} \pi_U(u_g)L$. 
\smallskip

It is enough to show that a vector in $H_U$ of the form $(\eta^{\tilde{\phi}}_{a})$, supported only at the $\phi$-coordinate, belongs to $\bigoplus_{g \in W} \pi_U(u_g)L$.
Since $B \subset A$ is regular, Theorem \ref{mainthm} implies that every $a \in A$ admits a decomposition $a = \sum_{g \in W} u_g\, E(u_g^* a)= \sum_{g \in W} u_g\, a_g,,$ where $\{u_g\}$ is a fixed set of (left) coset representatives of $W$, and we set $a_g := E(u_g^* a) \in B,$ the sum converging in the Bures topology on $A''$ with respect to $E''$. For a finite subset $F \subset W$, set $a_F := \sum_{g \in F} u_g\, a_g$. Then $a_F \to a$ in the Bures topology. Consider the vectors $(\eta^{\tilde{\phi}}_{a_g})$, which are supported only at the $\phi$-coordinate. Then
\begin{eqnarray}\nonumber
    &&\left\|(\eta^{\tilde{\phi}}_{a}) - \sum_{g \in F} \pi_U(u_g)\big((\eta^{\tilde{\phi}}_{a_g})\big)\right\|^2\\\nonumber
    &=& \left\|(\eta^{\tilde{\phi}}_{a - \sum_{g\in F} u_g a_g})\right\|^2\\\nonumber
    &=& \phi \circ E\big((a - \sum_{g\in F} u_g a_g)^*(a - \sum_{g\in F} u_g a_g)\big)\\\nonumber
    &=& w_{(\eta^{\tilde{\phi}}_{1}), (\eta^{\tilde{\phi}}_{1})} \circ E\big((a - a_F)^*(a - a_F)\big).
\end{eqnarray}
Since $w_{(\eta^{\tilde{\phi}}_{1}),(\eta^{\tilde{\phi}}_{1})}
=
\big\langle \,\cdot\, (\eta^{\tilde{\phi}}_{1}), (\eta^{\tilde{\phi}}_{1})\big\rangle \in B''_*$, it follows that the right-hand side converges to $0$.
Therefore, $(\eta^{\tilde{\phi}}_{a}) \in \bigoplus_{g \in W} \pi_U(u_g)L$, which implies $H_U \subset \bigoplus_{g \in W} \pi_U(u_g)L.$
Consequently, $H_U = \bigoplus_{g \in W} \pi_U(u_g)L.$

Now define an isomorphism 
$T : H_U \to \ell^2(W, L)$ by
$$
T\big(\sum_{g \in W} \pi_U(u_g)\xi_g\big) = \xi, \quad \text{where } \xi : W \to L \text{ is given by } \xi(g) = \xi_g.
$$
Now, for each $a \in A$, consider the map given by $\Theta(\pi_U(a)) = T \pi_U(a) T^*$. We claim that 
$\Theta(\pi_U(b)) = \pi_{\alpha}(\pi_U(b))$ for all $b \in B$, and 
$\Theta(\pi_U(u_g)) = \lambda^*_{g^{-1}}$ for all $g \in W$. Let $b \in B$. Then, for some $\xi \in \ell^2(W, L)$, we have
\begin{eqnarray}\nonumber
   \Theta(\pi_U(b))(\xi)&=& T \pi_U(b) (\sum_{g \in W}\pi_U(u_g)\xi_g)\\\nonumber
   &=& T  (\sum_{g \in W}\pi_U(bu_g)\xi_g)\\\nonumber
   &=& T  (\sum_{g \in W}\pi_U(u_gu^*_gbu_g)\xi_g)\\\nonumber
   &=& T  (\sum_{g \in W}\pi_U(u_g)\pi_U(\alpha_{g^{-1}}(b))\xi_g)\\\nonumber
   &=& \pi_{\alpha}(\pi_U(b))\xi
\end{eqnarray}
On the other hand, for any $g \in W$ and for some $\xi \in \ell^2(W, L)$, we have
\begin{eqnarray}\nonumber
    \Theta(\pi_U(u_g))(\xi) &=& T \pi_U(u_g) (\sum_{h \in W}\pi_U(u_h)\xi_h)\\\nonumber
    &=& T (\sum_{h \in W}\pi_U(u_gu_h)\xi_h)\\\nonumber
     &=& T (\sum_{h \in W}\pi_U(u_{gh} \sigma(h^{-1}, g^{-1})^*)\xi_h) \quad \text{ (by \cref{impinmainthm3})}\\\label{inmainthm}
     &=& T (\sum_{k \in W}\pi_U(u_{k}) \pi_U(\sigma(k^{-1}g, g^{-1})^*)\xi_{g^{-1}k}).
\end{eqnarray}
By by Remark \ref{impinmainthm6}$(ii)$, we know that $\lambda^*_{g^{-1}}=\pi_{\alpha}(\pi_U({\sigma(g,g^{-1})}^*))\lambda_g$. Hence, for some $\xi \in \ell^2(W, L)$, we have
\begin{eqnarray}\nonumber
    \pi_{\alpha}(\pi_U({\sigma(g,g^{-1})}^*))\lambda_g(\xi)(k)&=&\pi_U(\alpha_{k^{-1}}({\sigma(g,g^{-1})}^*))(\lambda_g(\xi)(k))\\\nonumber
    &=&\pi_U(\alpha_{k^{-1}}({\sigma(g,g^{-1})}^*))(\pi_U(\sigma(k^{-1},g))(\xi_{g^{-1}k}))\\\nonumber
    &=&\pi_U({\alpha_{k^{-1}}({\sigma(g,g^{-1})})}^*\sigma(k^{-1},g))(\xi_{g^{-1}k})\\\nonumber
    &=&\pi_U({{\sigma(k^{-1},e)}}{\sigma(k^{-1}g,g^{-1})}^*)(\xi_{g^{-1}k}) \quad \text{(by Remark \ref{impinmainthm6}$(i)$)}\\\label{inmainthm2} 
    &=&\pi_U({\sigma(k^{-1}g,g^{-1})}^*)(\xi_{g^{-1}k}).
\end{eqnarray}
Thus, from \cref{inmainthm} and \cref{inmainthm2}, we have $\Theta(\pi_U(u_g)) = \lambda^*_{g^{-1}}$ for all $g \in W$. Moreover, since $\Theta$ is an isometric $*$-algebra homomorphism, it follows that $\Theta(A) \subseteq B \rtimes^r_{(\alpha,\sigma)} W$. Now, since $A$ is the $C^*$-algebra generated by $\{u_g : g \in W\}$ and $B$, and $B \rtimes^r_{(\alpha, \sigma)} W$ is the $C^*$-algebra generated by $\{\pi_{\alpha}(\pi_U(b)) : b \in B\}$ and $\{\lambda_g : g \in W\}$, the surjectivity follows from the fact that $\Theta$ sends dense sets to dense sets. This completes the proof.
\qed
\end{prf}

\medskip
By \cite[Theorem 4.4]{CaSm2}, if $B$ is a unital simple $C^*$-algebra, $G$ is a discrete group, and $(B,G,\alpha,\sigma)$ is a twisted dynamical system such that $\alpha_g$ is outer for every $g \neq e$, then every intermediate $C^*$-algebra of
$A \subseteq B \rtimes^{\sigma}_{\alpha,r} G$ is of the form $B \rtimes^{\sigma}_{\alpha,r} H,$ where $H$ is a subgroup of $G$. As a consequence of Theorem~\ref{characterization-irreducible-regular}, we obtain the following corollary:
\begin{crlre}
Let $B \subset A$ be a regular, irreducible, unital inclusion of simple $C^*$-algebras, and let there exist a conditional expectation from $A$ onto $B$. Then, for every intermediate $C^*$-algebra $C$, the inclusion $B \subset C$ is regular.
\end{crlre}

\brmrk While the recent work of Palomares and Nelson \cite{PaN} provides a comprehensive categorical characterization of discrete inclusions using unitary tensor categories (UTCs), our approach focuses specifically on the group-theoretic structure of regular irreducible inclusions. We demonstrate that under the assumption of regularity, the symmetry is fully captured by the Weyl group $W$ acting as a reduced twisted crossed product. Unlike the categorical framework, our use of a generalized quasi-basis in the Bures topology allows for a concrete Fourier-like expansion of elements in $A$ directly via the Weyl group representatives. This provides a direct $C^{*}$-algebraic analog to the classical factor results of Choda (\cite{Ch2}), extending the structural theory beyond the finite-index regime without requiring the abstract machinery of tensor categories.\ermrk
We conclude this paper with few questions.
\medskip

\textbf{Question 1 (Crossed product realization).} 
Let $B \subset A$ be a regular unital inclusion of simple $C^*$-algebras admitting a faithful conditional expectation $E : A \to B$. Is it possible to characterize such inclusions in terms of group-theoretic data? More precisely, does there exist a discrete group $G$ and a (possibly twisted) action $(\alpha,\sigma)$ of $G$ on the $C^*$-algebra 
\[
\mathcal{C} = C^*\{B,\, B' \cap A\}
\]
such that
\[
A \cong \mathcal{C} \rtimes_r^{(\alpha,\sigma)} G?
\]

\medskip

\noindent
\textbf{Question 2 (Unitary orthonormal basis in the nonirreducible case).} 
Let $B \subset A$ be a regular (not necessarily irreducible) unital inclusion of simple $C^*$-algebras admitting a faithful conditional expectation $E : A \to B$. Does $E$ admit a unitary orthonormal basis?

\medskip
\section*{Acknowledgements}
The first author acknowledges the support of the grant ANRF/ECRG/2024/002328/PMS.

\bigskip

\noindent {\em Department of Mathematics and Statistics}\\
{\em Indian Institute of Technology Kanpur}\\
{\em Uttar Pradesh $208016$, India}
\medskip

\noindent {Keshab Chandra Bakshi:} bakshi209@gmail.com, keshab@iitk.ac.in\\
{Silambarasan C:} silamc23@iitk.ac.in\\
{Biplab Pal:} biplabpal32@gmail.com, bpal21@iitk.ac.in,

\end{document}